\documentclass{amsart}%
\usepackage{amsfonts}
\usepackage{amsmath}
\usepackage{amssymb}
\usepackage{graphicx}%
\providecommand{\U}[1]{\protect\rule{.1in}{.1in}}
\newtheorem{theorem}{Theorem}
\theoremstyle{plain}

\newtheorem{conjecture}{Conjecture}
\newtheorem{corollary}{Corollary}

\newtheorem{definition}{Definition}
\newtheorem{example}{Example}

\newtheorem{lemma}{Lemma}

\newtheorem{proposition}{Proposition}
\newtheorem{remark}{Remark}

\numberwithin{equation}{section}
\begin{document}
\title[Continuous path]{Stationary, Markov, stochastic processes with polynomial conditional moments
and continuous paths}
\author{Pawe\l \ J. Szab\l owski}
\address{Department of Mathematics and Information Sciences, \\
Warsaw University of Technology\\
ul Koszykowa 75, 00-662 Warsaw, Poland}
\email{pawel.szablowski@gmail.com}
\thanks{The author is very grateful to the unknown referee for pointing out misprints
and small mistakes as well as suggesting some simplifications in proofs. The
author is also grateful to the managing editor for indicating unknown to the
author positions of the literature.}
\date{May 2022}

\begin{abstract}
We are studying stationary random processes with conditional polynomial
moments that allow a continuous path modification. Processes with continuous
path modification, are important because they are relatively easy to simulate.
One does not have to care about the distribution of their jumps which is
always difficult to find. Among those processes with the continuous path are
the Ornstein-Uhlenbeck process, the Gamma process, the process with Arcsin or
Wigner margins and the Theta functions as the transition densities and others.
We give a simple criterion for the stationary process to have a continuous
path modification expressed in terms of skewness and excess kurtosis of the
marginal distribution.

\end{abstract}
\maketitle

\section{Introduction}

Inspired by the excellent paper of Ray \cite{Ray56} we decided to return to
the problem of path continuity of the stationary Markov processes. In the
meantime, there appeared new papers and notions. Among them is the class
stochastic processes with polynomial conditional moments. We confine ourselves
to the subclass of stationary Markov processes. We do so since within this
subclass we have identifiability of the marginal distribution, consequently
the existence of the set of polynomials that are orthonormal with respect to
this distribution and what is more we have the property that all its
conditional moments are polynomials in the conditioning random variable.

In fact, the theory of polynomial Markov processes (that is the Markov
processes with polynomial conditional moments) has two independent sources.
Starting with the paper \cite{Cuch12}, which was inspired by the applications
of the polynomial processes in mathematical finance, the theory of
$m-$polynomial processes started being developed. In this theory, the finite
$m$ leading to the assumption that the property of having polynomial
conditional moments is restricted only to polynomial of a degree less or equal
than $m.$

On the other hand, inspired by the series of papers \cite{BryWe},
\cite{brwe05}, \cite{BryWe10}, \cite{BryMaWe07}, \cite{BryMaWe11},
\cite{BryWe10}, \cite{BryWe12}, of W, Bryc, J. Weso\l owski and sometimes W.
Matysiak on quadratic harnesses, the author was impressed by the existence
within this theory the very useful families of polynomial martingales. And
thus wanted to examine the family of processes where such a family of
martingales appears. The restriction imposed by the assumptions of being a
harness or a quadratic harness would be imposed later, within the theory of
polynomial processes so far developed.

As shown in the series of papers \cite{SzablPoly}, \cite{SzabStac} and
\cite{SzabMarkov18} the assumption of polynomial conditional moments and the
assumption of Markovian stationarity lead to a very specific form of the
transition probability. Namely, these assumptions allow its Lancaster-type
expansion of the transition probability. As remarked recently in
\cite{Szabl21} all bivariate distributions that satisfy some mild technical
conditions and have the property that all their conditional moments are
polynomials in the conditioning random variable can be expanded in a
Lancaster-type series. The fact that we have such expansion at our disposal,
leads to a deeper insight into such processes. More precisely, we are able to
define families of polynomial and orthogonal martingales. Moreover, having a
guaranteed existence of all moments, we can refer to the simple Kolmogorov
theorem as the primary tool in examining the path continuity of the process
under consideration. Recall that the Kolmogorov continuity theorem assures the
path continuity provided certain conditions expressed in terms of some moments
of the bivariate distribution are satisfied. Of course, we lose generality and
the simplest statement concerning path continuity practically guarantees that
the paths of the process are of all of $r-$ th H\={o}lder class of continuity,
most commonly for $r<1/4$. Hence, we lose cases with the continuous path of at
most $r-$th H\={o}lder class for $r<1/4.$ But the simplicity of consideration
and conditions is the price. In my opinion worth to be paid. How much are we
getting? As mentioned above, very simple conditions that assure continuous
path modification, and due to an assumed form of the bivariate (i.e.,
consequently transitional) probability a deep insight into the structure and
behavior of the process.

We analyze several examples of marginal distribution. In some cases we prove
that they lead to continuous path modifications with $r-$ th H\={o}lder class
of continuity with $r<1/2$ and present related transitional densities. They
are the Normal case, which is known and presented for the completeness of the
paper and the Gamma case which is believed to be new. There are however two
examples concerning special cases of beta distribution, i.e., arcsine and
semicircle distributions which lead to transitional densities of the form of
some combinations of the Jacobi Theta function and families of orthogonal
martingales of the form (\ref{MO}) with coefficients $\alpha_{n}$ $\sim
n^{2}.$

The paper is organized as follows. First, we fix notation, then we recall some
facts from the theory of stochastic processes and the theory of probability.
In particular, we recall the notion of the stationary stochastic process with
polynomials conditional moments, the main object of the research presented in
this paper. Then we formulate some necessary conditions, expressed in terms of
the first four moments of the marginal distribution that allow continuous path
modification of the process with this given marginal. Finally, we present
examples of such stationary processes that allow continuous path modification.
These include Gaussian, Gamma, Laplace, Semicircle, Arcsine and $q-$Normal
processes. The name of the process refers naturally, to the name of the
marginal distribution.

\section{Notation and the basic facts}

Let us start with the following remarks concerning notation. We will be
considering only probability measures, that is, nonnegative measures that
integrate over their supports to $1$. Moreover, the integrals with respect to
such measure $\mu,$ will be exchangeably denoted either traditionally as $\int
f(x)d\mu(x)$ or as $Ef(X).$ Here we denote by $\mu$ the so-called distribution
of the random variable $X$. More precisely $\mu$ is defined by the following
formula:
\[
P(X\leq x)\allowbreak=\allowbreak\mu((-\infty,x])\allowbreak=:\allowbreak
\int_{-\infty}^{x}d\mu(x).
\]
In the above-mentioned formulae, $f:\operatorname*{supp}X\allowbreak
\rightarrow\allowbreak\mathbb{R},$ denoted a $\mu-$measurable function and
$\operatorname*{supp}$ denoted support of the random variable $X,$ which in
this case means support of its distribution. Since we will be considering
Markov stochastic processes the majority of measures considered will be at
most $2-$dimensional. Moreover, we will be often using the so-called tower
property in integration with respect to such $2-$dimensional measure. Namely,
if $X=(Y_{1},Y_{2})$ then
\begin{gather*}
Ef(Y_{1},Y_{2})=\int_{\operatorname*{supp}(X)}f(y_{1},y_{2})d\mu(y_{1}%
,y_{2})\\
=\int_{\operatorname*{supp}(Y_{2})}\int_{\operatorname*{supp}(Y_{1})}%
f(y_{1},.)d\mu_{Y_{1}|Y_{2}}(y_{1}|y_{2})d\lambda_{Y_{2}}(y_{2})=E(Ef(Y_{1}%
,Y_{2})|Y_{2})).
\end{gather*}
Here $\lambda_{Y_{2}}(y_{2})$ denotes a marginal measure of the random
variable $Y_{2}$ and $d\mu_{Y_{1}|Y_{2}}(y_{1}|y_{2})$ denotes the so-called
conditional measure of $Y_{1}$ given $Y_{2}\allowbreak=\allowbreak y_{2}.$ The
existence of such a measure is guaranteed by the theory of measure at
$\lambda_{Y_{2}}$- almost every point of $\operatorname*{supp}Y_{2}$.
Moreover, $Ef(Y_{1},Y_{2})|Y_{2})$ denotes the so-called conditional
expectation of a random variable $f(Y_{1},Y_{2})$ given $Y_{2}$.

There were four incentives to write this paper.

The first one is the so-called continuity Kolmogorov Theorem that reads the
following (see e.g. \cite{Strook79}, p.51):

\begin{theorem}
\label{Kolmog}Let $(S,d)$ be a metric space and let $X:[0,\infty)\times
\Omega\allowbreak\rightarrow\allowbreak S$ be a stochastic process. Suppose,
that for all $T>0$ there exist $3$ positive constants $\alpha,\beta,$ $K$ such
that $\forall0\leq s,t\leq T:$%
\[
Ed^{\gamma}\left(  X_{s},X_{t}\right)  \leq K\left\vert s-t\right\vert
^{1+\beta}.
\]
Then, there exists a modification $\tilde{X}$ of $X$ that has continuous
paths, $\forall t\geq0:P(X_{t}=\tilde{X}_{t})\allowbreak=\allowbreak1.$
Moreover, every path of $\tilde{X}$ is $\delta-$H\H{o}lder, for $\delta
\in(0,\beta/\gamma)$.
\end{theorem}

Note, that by a continuous mapping of time $s:[0,\infty)\allowbreak
\rightarrow\allowbreak\mathbb{R}$, like, for example, $s(x)\allowbreak
=\allowbreak\log x,$ $x\geq0,$ that doesn't affect the continuity of the paths
of the stochastic process, we can extend the formulation of the
above-mentioned theorem to the stochastic process defined for all real $t$.

In the sequel we will use the following notation concerning Gaussian
variables. Namely, if the vector $(X,Y)^{T}$ has bivariate normal distribution
with $EX\allowbreak=\allowbreak m_{1},$ $EY\allowbreak=\allowbreak m_{2},$
$\operatorname*{var}(x)\allowbreak=\allowbreak\sigma_{1}^{2}$,
$\operatorname*{var}(Y)\allowbreak=\allowbreak\sigma_{2}^{2}$ and
$\operatorname*{cov}(X,Y)\allowbreak=\allowbreak r$ then we will write
$\left(  X,Y\right)  ^{T}\allowbreak\sim$ \allowbreak$N(m_{1},m_{2};\sigma
_{1}^{2},\sigma_{2}^{2},r).$ In case of the one-dimensional Gaussian
distribution we write $X\sim N(m;\sigma^{2})$ when $EX\allowbreak=\allowbreak
m$ and $\operatorname*{var}(X)\allowbreak=\allowbreak\sigma^{2}.$ From now on
the symbol $\sim$ will also mean "has distribution" i.e. $X\sim\mu$ means that
$X$ has distribution defined by function $\mu$.

Hence, another incentive for writing this paper is the following auxiliary,
well-known result that we prove for the sake of completeness of the paper:

\begin{lemma}
\label{mom_d}Let $(X,Y)$ have bivariate Gaussian distribution $N(0,0;1,1,\rho
)$.

Then
\[
E(X-Y)^{2k}=\frac{(2k)!}{k!}(1-\rho)^{k}.
\]

\end{lemma}

\begin{proof}
[Proof suggested by the referee.]We know that if $(X,Y)$ is bivariate Normal
(Gaussian) then all its linear transformations. In particular $X-Y$ has Normal
distribution. Since $E(X-Y)\allowbreak=\allowbreak0$ and $E(X-Y)^{2}%
\allowbreak=\allowbreak2(1-\rho)$ by our assumptions, we deduce that $X-Y\sim
N(0;2(1-\rho)).$ Now recalling that if $X\sim N(0;\sigma^{2})$ then
$EX^{n}\allowbreak=\allowbreak\left\{
\begin{array}
[c]{ccc}%
1 & if & n=0\\
0 & if & n\text{ is odd}\\
(n-1)!! & if & n\text{ is even}%
\end{array}
\right.  $, we get: $E(X-Y)^{2k}\allowbreak=\allowbreak2^{k}(1-\rho
)^{k}(2k-1)!!\allowbreak=\allowbreak\frac{(2k)!}{k!}\allowbreak(1-\rho)^{k}$.
\end{proof}

Now, keeping in mind, that the Ornstein-Uhlenbeck process is the stationary
Gaussian process with $2-$dimensional density $N(0,0;1,1,\rho)$ with
$\rho\allowbreak=\allowbreak\exp(-\alpha t),$ for some positive $\alpha,$ we
see that
\[
E(X_{\tau}-X_{\tau+t})^{2k}\allowbreak=\allowbreak\frac{(2k)!}{k!}%
(1-\exp(-\alpha t))^{k}\cong F_{k}O(\left\vert t\right\vert ^{k}),
\]
as $t\allowbreak\rightarrow\allowbreak0$. Thus, the Ornstein-Uhlenbeck process
allows modification with the continuous path. This is an obvious fact since
Ornstein-Uhlenbeck process is a continuous transformation of the Wiener
process. It is however not so obvious that the paths are of the $\gamma
-$H\H{o}lder class with $\gamma<\inf_{k}\frac{k-1}{2k}\allowbreak
=\allowbreak1/2.$

The third incentive for writing this paper were the following two results. The
first one is the so-called formula for the expansion of the $2-$dimensional
distribution $dG(x,y)$ of, say $(X_{\tau},X_{\tau+t}),$ where $X_{\tau}$ and
$X_{\tau+t}$ belong to some normalized Ornstein-Uhlenbeck process. Namely, we
have the following Lancaster-type expansion%
\begin{equation}
dG(x,y)=\frac{1}{2\pi}\exp(-\frac{x^{2}+y^{2}}{2})\sum_{j=0}^{\infty}%
H_{j}(x)H_{j}(y)\exp(-j\alpha t)/j!dxdy,\label{expG}%
\end{equation}
where $H_{j}(x)/\sqrt{j!}$ are ortonormal with respect to the measure with the
density $\frac{1}{\sqrt{2\pi}}\exp(-x^{2}/2)$. (\ref{expG}) is a simple
modification of the so-called Mehler expansion given e.g., in \cite{Sriv72},
formula (1.2). Now, we can easily deduce that this example has a nice feature
that for every $k$, conditions for having a continuous path modification are
expressed in terms of moments of the marginal distribution. One knows that
$H_{j}$ are the so-called monic (that is with $1$ as the coefficient by
$x^{j}$) Hermite polynomials of the probabilistic type.

The fourth result that spurred to write this paper is slightly more
complicated and requires a small introduction, but is basically simple and
concerns, so to say, a generalization of expansion (\ref{expG}).

As stated above, in this paper we will be examining the path continuity of the
processes defined on the whole real line. But out of all defined so stochastic
processes, we will confine our considerations to stationary stochastic
processes additionally having the property of possessing polynomial
conditional moments. The class of Markov stochastic processes having
polynomial conditional moments has been described and analyzed in the series
of papers \cite{SzablPoly}, \cite{SzabStac} and \cite{SzabMarkov18}. Recently
yet another property of such class of stochastic processes has been added.
Namely, in \cite{Szabl21} it has been shown that under some regularity
conditions, the two-dimensional distributions of a Markov stochastic process
with the property that all its conditional moments are polynomials of the
conditional random variable, must be of the Lancaster-type, that will be
explained and defined below. The example of a Lancaster-type distribution is
given by (\ref{expG}), above. The term Lancaster-type refers to the series of
papers of H.O. Lancaster \cite{Lancaster58}, \cite{Lancaster63(1)},
\cite{Lancaster63(2)}, \cite{Lancaster75}, where those types of expansions
were introduced and studied.

Since there are several points in fixing notation and exposing necessary
assumptions that will enable necessary regularity, let's present them first.

Let $\mathbf{X}=(X_{t})_{t\in\mathbb{\mathbb{R}}}$ be a real stochastic
process defined on some probability space $(\Omega,\mathcal{F},P)$. By the
stationary Markov processes we mean those Markov processes $\mathbf{X}%
=(X_{t})_{t\in\mathbb{R}}$ that have marginal distributions that do not depend
on the time parameter and the property that the conditional distributions of
say $X_{t}$ given $X_{s}$ does depend only on $t-s.$

We will assume that $\forall n\in N,$ $t\in\mathbb{\mathbb{R}}$ :
$E|X_{t}|^{n}<\infty.$ More precisely, we assume that the distributions of
$X_{t}$ will be identifiable by their moments. This assumption is a slightly
stronger assumption than the existence of all moments. For example, it is
known that if $\exists\beta>0:\int\exp(\beta|x|)d\mu(x)<\infty\text{,}$ then
the measure $\mu$ is identifiable by its moments. Here $\mu$ denotes the
distribution of $X_{0}.$ In fact, there exist other conditions assuring this.
For details see e.g. \cite{Sim98}.

In \cite{SzabStac} one considers the general case of the cardinality of
$\operatorname*{supp}\mu.$ But in order to avoid unnecessary complications, we
will confine ourselves to the infinite number of points of the set
$\operatorname*{supp}\mu$ , i.e., infinite cardinality of
$\operatorname*{supp}\mu$.

To fix further notation, let us denote $\mathcal{F}_{\leq s}=\sigma(X_{r}%
:r\in(-\infty,s]\cap\mathbb{\mathbb{R}})$, $\mathcal{F}_{\geq s}=\sigma
(X_{r}:r\in\lbrack s,\infty)\cap\mathbb{\mathbb{R}})$ and $\mathcal{F}%
_{s,u}=\sigma(X_{r}:r\notin(s,u),r\in\mathbb{\mathbb{R}})$.

Let us also denote by $\mu(.)$ and by $\eta(.|y,\tau)$ respectively marginal
stationary distribution and transition distribution of our Markov process.
That is $P(X_{t}\in A)=\int_{A}\mu(dx)$ and $P(X_{t+\tau}\in A|X_{t}%
=y)=\int_{A}\eta(dx|y,\tau)$. Stationarity of $\mathbf{X}$ means thus that
$\forall\mathbb{T}\ni\tau\neq0,B\in\mathcal{B}$ $\mathcal{(B}$ denotes here
Borel $\sigma-$field)
\[
\mu(B)=\int\eta(B|y,\tau)\mu(dy).
\]

By $L_{2}(\mu)$ let us denote the space spanned by the real functions that are
square-integrable (more precisely by the set of equivalence classes) with
respect to $\mu$ i.e.
\[
L_{2}(\mu)=\{f:\mathbb{\mathbb{R}}\longrightarrow\mathbb{\mathbb{R}}%
,\int|f|^{2}d\mu<\infty\}.
\]
Our assumption of the existence of all moments and their ability to define the
unique underlying measure (uniqueness of the related moment problem) of
$X_{0}$ in terms of $L_{2}(\mu),$ implies that there exists a set of
orthogonal polynomials that constitute the orthogonal base of this space. This
statement is based on the properties of real Hilbert spaces found in any
textbook on mathematical analysis or in particular in \cite{Sim98} or
\cite{Alexits61}. Let us denote these polynomials by $\{h_{n}\}_{n\geq-1}.$
Additionally, let us assume that polynomials $h_{n}$ are orthonormal and
$h_{-1}(x)=0,$ $h_{0}(x)=1.$ Thus we will assume that for all $i,j\geq0:$
\begin{equation}
\int h_{i}(x)h_{j}(x)d\mu(x)\allowbreak=\allowbreak\delta_{ij}, \label{or}%
\end{equation}
where, as usually, $\delta_{ij}$ denotes Kronecker's delta.

Having introduced $\mathcal{F}_{\leq s}$, $\mathcal{F}_{\geq s}$ and
$\mathcal{F}_{s,u}$ we can specify more precisely the classes of Markov
processes that we are considering in this paper. Processes with polynomial
conditional moments are those for which the following condition holds for all
$n\geq0$ and $s\in\mathbb{R}$%
\[
E(X_{t}^{n}|\mathcal{F}_{\leq s})=P_{n}(X_{s}|t,s)\text{\ \ }\;\text{a.s.,}%
\]
where $P_{n}$ denotes here polynomial of degree not exceeding $n$ in $X_{s}$
with coefficients depending on $t$ and $s$.

Note that if the above-mentioned condition holds only for $n\allowbreak
=\allowbreak1$ and $P_{1}(x|t,s)\allowbreak=\allowbreak x$ then we deal with a martingale.

Occasionally, there will appear processes that are also harnesses. So let us
define a $n-$harness by the following condition:%
\[
E(X_{t}^{n}|\mathcal{F}_{s,u})=R_{n}(X_{s},X_{u}|s,t,u)~~\text{a.s.,}%
\]

where $R_{n}$ is a polynomial of degree not exceeding in $X_{s}$ and $X_{u}$
with coefficients depending on $s,t,$ and $u$.

Hammersley in 1967 introduced the notion of harness in the paper \cite{Ham67}
by considering $1-$harnesses. W, Bryc, J. Weso\l owski and W. Matysiak
considered quadratic harnesses that are both $1$- and $2-$harnesses according
to the above-mentioned definition.

Thus, the class of Markov processes that we consider, is a class of stochastic
processes that satisfies some mild technical assumptions that were described
and interpreted in \cite{SzabStac} and moreover satisfying the following
conditions: $\forall t\in\mathbb{T},n\in N:E(X_{t}^{n})=m_{n}$ and $\forall
n\geq1,s<t:$%

\begin{equation}
E(X_{t}^{n}|\mathcal{F}_{\leq s})=Q_{n}(X_{s},t-s)\text{\ \ }\;\text{a.s.,}%
\label{p_reg}%
\end{equation}
where $Q_{n}(x,t-s)$ is a polynomial of order not exceeding $n$ in
$x\text{.}\;$

It has been shown in \cite{SzabStac} that under the above-mentioned regularity
assumptions and also under the following assumption that $\eta\ll\mu$ and
\begin{equation}
\int\left(  \frac{d\eta}{d\mu}\right)  ^{2}d\mu<\infty, \label{RNSt}%
\end{equation}
where, as above, $\mu(dx)$ and $\eta(dx|y,t)$ denote respectively marginal and
transitional measures of $X,$ the following expansion holds
\begin{equation}
\frac{d\eta}{d\mu}\left(  x|y,t\right)  =\sum_{n\geq0}\exp(-\alpha_{n}%
t)h_{n}(x)h_{n}(y). \label{gest}%
\end{equation}
In this formula, we will set $\alpha_{0}\allowbreak=\allowbreak0$ and there
appear certain positive constants $\left\{  \alpha_{i}\right\}  _{i\geq1}$
whose existence is guaranteed by the mentioned above technical assumptions.
The constants $\left\{  \alpha_{i}\right\}  _{i\geq1}$ allow to define
orthogonal martingales defined by the formula:
\begin{equation}
M_{n}(X_{t},t)=\exp(\alpha_{n}t)h_{n}(X_{t}),\;\text{{}}\;\;n\geq1. \label{MO}%
\end{equation}

The class of such stationary stochastic process will be briefly called SMPR.
More precisely, since from (\ref{gest}) it follows that such processes are
completely characterized by the distribution $\mu$ and a sequence $\left\{
\alpha_{n}\right\}  $ of positive numbers. We will write to denote such a
process $\mathbf{X\allowbreak=\allowbreak}\left\{  X_{t}\right\}
_{t\in\mathbb{R}}\allowbreak=\allowbreak SMPR\left(  \left\{  \alpha
_{n}\right\}  ,\mu\right)  $.

Thus, under the above-mentioned assumption, the two-dimensional distribution
of say $\left(  X_{\tau},X_{\tau+t}\right)  $ is given by the formula
\begin{equation}
\rho(dx,dy)=d\mu(x)d\mu(y)\sum_{n\geq0}\exp(-\alpha_{n}t)h_{n}(x)h_{n}(y).
\label{Lanc}%
\end{equation}

\begin{remark}
\label{cond}Notice, that from the fact that $M_{n}(X_{t,}t)$ is a martingale,
it follows that
\[
E(M_{n}(X_{t+\tau},t+\tau)|\mathcal{F}_{\leq\tau})\allowbreak=\allowbreak
M_{n}(X_{\tau},\tau),
\]
hence, following (\ref{MO}), we see that
\[
E\left(  h_{n}(X_{t+\tau})|\mathcal{F}_{\leq\tau}\right)  \allowbreak
=\allowbreak\exp(-\alpha_{n}t)h_{n}(X_{\tau}),
\]
a.s. mod $d\mu$.
\end{remark}

Now, let us recall that in \cite{SzabChol} the following numbers $\left\{
c_{j,k}\right\}  _{j\geq0,0\leq n\leq j}$ were introduced and analyzed. Their
interpretation is the following:
\begin{equation}
x^{j}\allowbreak=\allowbreak\sum_{n=0}^{j}c_{j,n}h_{n}(x). \label{xnah}%
\end{equation}
for all $j\geq0$. Let us set $c_{j,n}\allowbreak=\allowbreak0$ for $n>j.$ Let
us also denote by $L_{k}$ the following lower-triangular matrix $[c_{j,n}%
]_{j=0,\ldots,k,n=0,\ldots k}.$ It has been remarked in \cite{SzabChol}
(Propositions 1 and 2) that
\[
L_{k}L_{k}^{T}=\mathbf{M}_{k},
\]
where $\mathbf{M}_{k}$ is the moment matrix, i.e., $\mathbf{M}_{k}%
=[m_{i+j}]_{i=0,\ldots,k,j=0,\ldots,k},$ where $m_{j}\allowbreak
=\allowbreak\int x^{j}d\mu(x)$ that is equal to $j-$th moment of the
distribution $\mu$. Hence, the coefficients $c_{j,n}$ can be computed directly
from the moments' matrix. Besides, we know by (\cite{SzabChol}, Proposition
1(iii)):%
\[
\sum_{n=0}^{\min(j,k)}c_{j,n}c_{k,n}=m_{j+k}.
\]
for all $j,k\geq0.$ We have the following observation:

\begin{remark}%
\begin{align}
E(X_{t+\tau}^{j}|\mathcal{F}_{\leq\tau})  &  =E(\sum_{n=0}^{j}c_{j,n}%
h_{n}(X_{t+\tau})|\mathcal{F}_{\leq\tau})=\sum_{n=0}^{j}c_{j,n}\exp
(-\alpha_{n}t)h_{n}(X_{\tau})\label{condmom}\\
&  =X_{\tau}^{j}-\sum_{n=1}^{j}c_{j,n}(1-\exp(-\alpha_{n}t))h_{n}(X_{\tau
}),\nonumber
\end{align}
mod $d\mu,$ since we have $\alpha_{0}\allowbreak=\allowbreak0.$
\end{remark}

In the sequel, we will use the following almost trivial lemma. It has been
presented in assertion 6 of Proposition 1 in \cite{SzabDif22}. Since it is
important in the present context, we will present its generalized version once
more with its simple proof.

\begin{lemma}
\label{momenty}Suppose that a sequence $\left\{  b_{n}\right\}  _{n\geq0}$ is
a positive moment sequence sequence.

i) Suppose that $b_{0}\allowbreak=\allowbreak1$, and $b_{4m+2}\allowbreak
=\allowbreak b_{2m+1}^{2}$ for some $m\geq0$. Then $\forall n\geq0:$
$b_{n}\allowbreak=\allowbreak b_{2m+1}^{n/(2m+1)}$.

ii) Suppose that $b_{0}\allowbreak=\allowbreak1$ and $b_{4m}\allowbreak
=\allowbreak b_{2m}^{2}$ for some $m\geq1.$ Let us set $p\allowbreak
=\allowbreak(b_{1}\allowbreak+\allowbreak b_{2m}^{1/(2m)})/(2b_{2m}^{1/(2m)})$
Then for all $n\geq0:b_{2n}\allowbreak=\allowbreak b_{2m}^{n/m}$ and
$b_{2n+1}\allowbreak=\allowbreak b_{2m}^{(2n+1)/(2m)}p-b_{2m}^{(2n+1)/(2m)}%
(1-p).$
\end{lemma}

\begin{proof}
Let $X$ denote a random variable whose moments are $b_{n},$ that is
$EX^{n}\allowbreak=\allowbreak b_{n}.$ In the case i) we have $E(X^{4m+2}%
)\allowbreak=\allowbreak(E(X^{2m+1}))^{2}.$ That is $E(X^{2m+1}-EX^{2m+1}%
)^{2}\allowbreak=\allowbreak\operatorname*{var}(X^{2m+1})\allowbreak
=\allowbreak0.$ But this equality means that the distribution of $X$ is a
one-point distribution, i.e. $P(X=b_{2m+1}^{1/(2m=1)})\allowbreak
=\allowbreak1.$ In the case ii) we deduce that $E(X^{2m}-EX^{2m}%
)^{2}\allowbreak=\allowbreak0.$ That is that the distribution of $X^{2}$ is a
one point distribution. One can find that $X^{2}\allowbreak=\allowbreak
b_{2m}^{1/(2m)}$ and that $P(X=b_{2m}^{1/(2m)})\allowbreak=\allowbreak
p\allowbreak=\allowbreak1-P(X=-b_{2m}^{1/(2m)})$, for some $p\in\lbrack0,1].$
But it means that $b_{1}\allowbreak=\allowbreak EX\allowbreak=\allowbreak
b_{2m}^{1/(2m)}p\allowbreak-\allowbreak b_{2m}^{1/(2m)}(1-p).$ Parameter $p$
can be be found to be equal to $(b_{1}\allowbreak+\allowbreak b_{2m}%
^{1/(2m)})/(2b_{2m}^{1/(2m)})$. Then and $b_{2n+1}$ is a $2n+1-$th moment of
such variable that is $b_{2m}^{(2n+1)/(2m)}p\allowbreak-\allowbreak
b_{2m}^{(2n+1)/(2m)}(1-p)$.
\end{proof}

\begin{remark}
1) It has been remarked in [\cite{Szabl21}, Corollary 1) that the sequence
$\left\{  \exp(\alpha_{n}t)\right\}  _{n\geq0}$ must be such that 1)
$\sum_{n\geq0}\exp(-2\alpha_{n}t)<\infty$ for all $t\in\mathbb{R}$. 2) If
$\operatorname*{supp}\mu$ is unbounded, then $\left\{  \exp(-\alpha
_{n}t)\right\}  _{n\geq0}$ must be a moment sequence. Moreover, since all
$\left\{  \alpha_{n}\right\}  _{n\geq1}$ are positive then we deduce that the
support of the measure with respect to which $\left\{  \exp(-\alpha
_{n}t)\right\}  _{n\geq0}$ is a moment sequence must have support contained in
$[0,1]$. Consequently, not only the matrix $\left\{  \exp(-\alpha
_{i+j}t)\right\}  _{0\leq i,j,\leq n}$ but also the matrix $\left\{
\exp(-\alpha_{i+j+1}t)\right\}  _{0\leq i,j,\leq n}$ must be nonnegative
definite. In particular, we deduce that
\[
~2\alpha_{2n+1}\geq\alpha_{2n}+\alpha_{2n+2}\text{ and }2\alpha_{2n}\geq
\alpha_{2n-1}+\alpha_{2n+1}.
\]
But this means that for all $m\geq0$ we have
\[
2\alpha_{m}\geq\alpha_{m-1}+\alpha_{m+1},
\]
which means that the sequence $\left\{  \alpha_{n}\right\}  _{n\geq0}$ must be
a concave sequence. We must also have for all $i,n\geq0,$ $i+n\geq1$
\[
\alpha_{2n}+\alpha_{2i}\leq2\alpha_{n+i}\text{ and }\alpha_{2i+1}%
+\alpha_{2i+2n+1}\leq2\alpha_{2i+2n}.
\]
The last inequality follows the fact that for $\left\{  \exp(-\alpha
_{n}t)\right\}  _{n\geq0}$ being a moment sequence is equivalent to the fact
that for all $m$ the matrices $\left[  \exp(-\alpha_{i+j}t)\right]  _{0\leq
i,j\leq m}$ $\left[  \exp(-\alpha_{i+j+1}t)\right]  _{0\leq i,j\leq m}$must be
nonnegative defined. Thus, in particular, all central $2\times2$ minors must
be nonnegative. Now some such minors are the following ones: $\left[
\begin{tabular}
[c]{ll}%
$\exp(-t\alpha_{2i}$ & $\exp(-\alpha_{2i+m}t)$\\
$\exp(-\alpha_{2i+m}t)$ & $\exp(-\alpha_{2m+2i}t)$%
\end{tabular}
\ \ \ \right]  $ and $\left[
\begin{tabular}
[c]{ll}%
$\exp(-t\alpha_{2i+1}$ & $\exp(-\alpha_{2i+2n}t)$\\
$\exp(-\alpha_{2i+2n}t)$ & $\exp(-\alpha_{2n+2i+1}t)$%
\end{tabular}
\ \ \ \right]  $. In the first of these matrices we set $n\allowbreak
=\allowbreak i+m.$
\end{remark}

\begin{remark}
\label{a2a}Notice also that if the support of the measure $\mu$ is unbounded
and we are able to show that $\alpha_{2}\allowbreak=\allowbreak2\alpha
_{1\text{.}}$ then, as it follows from assertion $6$ of Proposition $1.$ we
must have $\alpha_{n}\allowbreak=\allowbreak n\alpha_{1}$ for all $n\geq1$.
Hence, consequently, $SMPR\left(  \left\{  \alpha_{n}\right\}  ,\mu\right)  $
is additionally a harness (for details see \cite{SzabStac}).
\end{remark}

\section{Necessary conditions}

Having said this, we can state that the paper is dedicated to defining
conditions under which a given $SMPR\left(  \left\{  \alpha_{n}\right\}
,\mu\right)  $ allows continuous path modification.

Thus we can calculate the quantities $E\left\vert X_{\tau}-X_{\tau
+t}\right\vert ^{\alpha}$ that are needed in order to apply Kolmogorov's
continuity theorem and examine the dependence of these quantities on $t.$

Of course, in order to simplify calculations, we consider only $\alpha
\allowbreak=\allowbreak2k$ for some natural $k.$

It turns out that for further analysis, we will need the following moments
defined by the formula%
\begin{equation}
\int x^{j}h_{n}(x)d\mu(x)\allowbreak=\allowbreak EX^{j}h_{n}(X).\label{Hjn}%
\end{equation}
Notice that we have
\[
EX^{j}h_{n}(X)=c_{j,n},
\]
where $c_{j,n}$ are defined above, by (\ref{xnah}). Let us observe that from
the definition of the orthogonal polynomials it follows that $\forall0\leq
j<n$ we have $EX^{j}h_{n}(X)=0.$ Thus, in the sequel we will be using
parameters $c_{j,n}$ set as $0$ for $j<n.$

\begin{remark}
\label{og}Notice that regardless of the boundedness of the support of marginal
measure $\mu$ we observe that the following sequence
\begin{equation}
\left\{  m_{2j}-\sum_{n=1}^{j}\left(  1-\exp(-\alpha_{n}t)\right)  c_{j,n}%
^{2}\right\}  _{j\geq0} \label{bimom}%
\end{equation}
must be a moment sequence for every $t>0$. This simple observation follows the
fact that
\[
E(X_{\tau}X_{\tau+t})^{j}=\int\int(xy)^{j}\rho(dx,dy),
\]
where $\rho(dx,dy)$ is given by (\ref{Lanc}). Then we use the definition of
coefficients $c_{j,n}$ and the fact that $\sum_{n=0}^{j}c_{j,n}^{2}%
\allowbreak=\allowbreak m_{2j}$.
\end{remark}

Now we can formulate sufficient conditions for the process
$\mathbf{X\allowbreak=\allowbreak}SMPR\left(  \left\{  \alpha_{n}\right\}
,\mu\right)  $ to allow continuous path modification. Namely, we have the
following lemma.

\begin{lemma}
\label{ciag}Let $\mathbf{X\allowbreak=\allowbreak}\left\{  X_{t}\right\}
_{t\in\mathbb{R}}\allowbreak=\allowbreak SMPR\left(  \left\{  \alpha
_{n}\right\}  ,\mu\right)  $ with polynomial $\left\{  h_{n}(x)\right\}  $
orthonormal with respect to $\mu.$ Let us define numbers $\left\{
c_{j,n}\right\}  _{j\geq0,0\leq n\leq j}$ by (\ref{Hjn}). If for some $k\geq1$
and $r>1$ the following conditions hold:%
\begin{equation}
\sum_{n=1}^{2k}\alpha_{n}^{s}\sum_{j=n}^{2k}(-1)^{j}\binom{2k}{j}%
c_{j,n}c_{2k-j,n}=0, \label{rown}%
\end{equation}
for $s\allowbreak=\allowbreak1,\ldots,r,$ then the process $\mathbf{X}$ allows
continuous path modification and the so-modified process has paths of the
$m-$th Holder class where $m<\frac{r}{2k}.$
\end{lemma}

\begin{proof}
Let us start with the following calculation:%
\begin{gather}
E(X_{\tau+t}-X_{\tau})^{m}=\sum_{j=0}^{m}(-1)^{j}\binom{m}{j}EX_{\tau}%
^{m-j}E(X_{t+\tau}^{j}|\mathcal{F}_{\leq\tau})=\label{expan}\\
\sum_{j=0}^{m}(-1)^{j}\binom{m}{j}EX_{\tau}^{m-j}(X_{\tau}^{m}-\sum_{n=1}%
^{j}(1-\exp(-\alpha_{n}t))c_{j,n}h_{n}(X_{\tau})=\nonumber\\
-\sum_{j=0}^{m}(-1)^{j}\binom{m}{j}\sum_{n=1}^{j}(1-\exp(-\alpha
_{n}t))c_{m-j,n}c_{j,n}\nonumber\\
=-\sum_{n=1}^{m}(1-\exp(-\alpha_{n}t))\sum_{j=n}^{m}(-1)^{j}\binom{m}%
{j}c_{j,n}c_{m-j,n}\nonumber\\
=\sum_{k=1}^{r}(-1)^{k-1}t^{k}/k!\sum_{n=1}^{m}\alpha_{n}^{k}\sum_{j=n}%
^{m}(-1)^{j}\binom{m}{j}c_{j,n}c_{m-j,n}+O(t^{r+1})\text{.}\nonumber
\end{gather}

Now we take $m\allowbreak=\allowbreak2k$ and see that $E\left\vert X_{t+\tau
}-X_{\tau}\right\vert ^{2k}\allowbreak\cong\allowbreak O(\left\vert
t\right\vert ^{r+1}).$ Now we can apply Kolmogorov's Theorem \ref{Kolmog}.
\end{proof}

Notice that, the crucial from the point of view of this paper, numbers
$c_{j,n}$ can be expressed in terms of moments of the marginal distribution
$\mu.$ Since the first two particular cases of $k$ are the most important let
us find numbers $c_{j,n}$ for $n\allowbreak=\allowbreak0,1,2$ and
$j\allowbreak=\allowbreak0,1,2,3,4.$The conditions presented in (\ref{rown})
are somewhat difficult to satisfy in the general case and have rather a
theoretical character. The most important case is the case $k\allowbreak
=\allowbreak2.$ Notice that then we have only two constants $\alpha_{1}$ and
$\alpha_{2}$ involved. They are thus the most important from the point of view
of the continuity of the paths of the analyzed process. Hence, let us analyze
this case in more detail.

\begin{proposition}
Let $X\sim\mu.$ Let us denote $EX\allowbreak=\allowbreak\nu$, $m_{j}%
\allowbreak=\allowbreak E(X-\nu)^{j},$ and $h_{j}(x)$ the polynomial of degree
$j$ that is orthonormal with respect to the measure $\mu$. Let the numbers
$c_{j,n}$ be defined by (\ref{Hjn}). Then:

i) $c_{j,0}\allowbreak=\allowbreak EX^{j}\allowbreak=\allowbreak\sum_{k=0}%
^{j}\binom{j}{k}\nu^{k}m_{j-k},$ with an obvious fact that $m_{1}%
\allowbreak=\allowbreak0.$

ii) $c_{1,1}\allowbreak=\allowbreak\sqrt{m_{2}},$ $c_{2,1}\allowbreak
=\allowbreak(m_{3}+2\nu m_{2})/\sqrt{m_{2}}$, $c_{3,1}\allowbreak
=\allowbreak(m_{4}+3\nu m_{3}+3\nu^{2}m_{2})/\sqrt{m_{2}},$ consequently we
have $E(X_{\tau}-X_{\tau+t})^{2}\allowbreak=\allowbreak2m_{2}(1-\exp
(-\alpha_{1}t)).$

iii) $c_{2,2}\allowbreak=\allowbreak\sqrt{m_{4}m_{2}-m_{3}^{2}-m_{2}^{3}%
}/\sqrt{m_{2}},$ consequently we have
\begin{gather*}
E(X_{\tau}-X_{\tau+t})^{4}\allowbreak=\allowbreak2(m_{4}+3m_{2}^{2}%
)\allowbreak-\allowbreak2\exp(-\alpha_{1}t)(4m_{2}m_{4}-3m_{3}^{2}%
)/m_{2}\allowbreak\\
+\allowbreak6\exp(-\alpha_{2}t)\left(  m_{2}m_{4}-m_{3}^{2}-m_{2}^{3}\right)
/m_{2}.\allowbreak
\end{gather*}

\end{proposition}

\begin{proof}
i) Since $h_{0}(x)\allowbreak=\allowbreak1,$ we must have $c_{j,0}%
\allowbreak=\allowbreak EX^{j}h_{0}(X)\allowbreak=\allowbreak E(X-\nu+\nu
)^{j}\allowbreak=\allowbreak\sum_{k=0}^{j}\binom{j}{k}\nu^{k}m_{j-k}$ with
$c_{1,0}\allowbreak=\allowbreak\nu$. Moreover, we have following
(\ref{expan})
\[
E(X_{\tau}-X_{\tau+t})^{2}=2c_{2,0}-2(c_{1,0}^{2}+\exp(-\alpha_{1}%
t)m_{2})=2m_{2}-2m_{2}\exp(-\alpha_{1}t).
\]
ii) We start with the well-known fact that polynomials that are orthogonal
with respect to the measure $\mu$ are given by the formula%
\begin{equation}
\eta_{n}(x)\allowbreak=\allowbreak a_{n}\det%
\begin{bmatrix}
1 & EX & \ldots & EX^{n}\\
EX & EX^{2} & \ldots & EX^{n+1}\\
\ldots & \ldots & \ldots & \ldots\\
1 & x & \ldots & x^{n}%
\end{bmatrix}
, \label{port}%
\end{equation}
for some constants $a_{n}$. Hence, orthonormal polynomials $h_{n}(x)$ are
given as $b_{n}\eta_{n}(x)$ for suitably chosen constants $b_{n}$. Thus we
have $h_{1}(x)\allowbreak=\allowbreak b_{n}\det%
\begin{bmatrix}
1 & \nu\\
1 & x
\end{bmatrix}
\allowbreak=\allowbreak b_{1}(x-\nu).$ Further, we have to have $Eh_{1}%
^{2}(X)\allowbreak=\allowbreak1,$ so $b_{1}\allowbreak=\allowbreak
1/\sqrt{m_{2}}.$ Hence, we have $c_{j,1}\allowbreak=\allowbreak(EX^{j}%
(X-\nu))/\sqrt{V}\allowbreak=\allowbreak(EX^{j+1}-\nu EX^{j})/\sqrt{m_{2}}$.
In particular, we have: $c_{1,1}\allowbreak=\allowbreak\sqrt{m_{2}},$
$c_{2,1}\allowbreak=\allowbreak\allowbreak(m_{3}+2cm_{2})/\sqrt{m_{2}}$ and
$c_{3,1}\allowbreak=\allowbreak(m_{4}+3\nu m_{3}+3\nu^{2}m_{2})/\sqrt{m_{2}}.$

iii) Using (\ref{port}), we get
\begin{align*}
h_{2}(x)\allowbreak &  =\allowbreak b_{2}\det%
\begin{bmatrix}
1 & \nu & m_{2}+c^{2}\\
\nu & m_{2}+\nu^{2} & m_{3}+3\nu m_{2}+\nu^{3}\\
1 & x & x^{2}%
\end{bmatrix}
\allowbreak\\
&  =\allowbreak b_{2}(m_{2}(x^{2}-m_{2}-\nu^{2})\allowbreak-\allowbreak
(m_{3}+2\nu m_{2})(x-\nu).
\end{align*}
Since we are interested in orthonormal $h_{2}$ we have to calculate
$Eh_{2}(X)^{2}.$ Simple, but the lengthy calculation gives
\[
Eh_{2}(X)^{2}=b_{2}^{2}m_{2}(m_{2}m_{4}-m_{3}^{3}-m_{2}^{3}).
\]
Hence, we have $b_{2}\allowbreak=\allowbreak1/\left(  \sqrt{m_{2}(m_{2}%
m_{4}-m_{3}^{3}-m_{2}^{3})}\right)  .$ Now, we have $c_{2,2}\allowbreak
=\allowbreak b_{2}EX^{2}h_{2}(X)\allowbreak=\allowbreak b_{2}(m_{2}m_{4}%
-m_{3}^{3}-m_{2}^{3})\allowbreak=\allowbreak\sqrt{(m_{2}m_{4}-m_{3}^{3}%
-m_{2}^{3})}/\sqrt{m_{2}}.$
\end{proof}

\begin{corollary}
Under the assumptions about Lemma \ref{ciag}, and upon applying Kolmogorov
Theorem \ref{Kolmog} we get%

\[
E(X_{\tau}-X_{\tau+t})^{4}=t(\alpha_{1}(8m_{4}m_{2}-6m_{3}^{2})/m_{2}%
-6\alpha_{2}(m_{2}m_{4}-m_{3}^{2}-m_{2}^{3}))+O(t^{2}).
\]
Consequently, the process allows continuous path modification if only
\[
\frac{4m_{4}m_{2}-3m_{3}^{2}}{3(m_{4}m_{2}-m_{3}^{2}-m_{2}^{3})}>0
\]
and%
\begin{equation}
\alpha_{2}=\alpha_{1}\frac{4m_{4}m_{2}-3m_{3}^{2}}{3(m_{4}m_{2}-m_{3}%
^{2}-m_{2}^{3})}. \label{a2}%
\end{equation}
Now let's introduce, popular in mathematical statistics, parameters $\kappa
-$kurtosis (excess kurtosis more precisely) defined as $\kappa\allowbreak
=\allowbreak m_{4}/m_{2}^{2}-3$ and Fisher's skewness parameter $s$ defined as
$s\allowbreak=\allowbreak m_{3}/m_{2}^{3/2}$. (\ref{a2}) takes now, the
following form:%
\begin{equation}
\alpha_{2}\allowbreak=\allowbreak\alpha_{1}\frac{12+4\kappa-3s^{2}}%
{6+3\kappa-3s^{2}}. \label{a2m}%
\end{equation}

\end{corollary}

Since the ratio of $\alpha_{2}/\alpha_{1}$ is important in sorting off
stationary processes which do not allow continuous path modification (at least
as far as our simple theory is concerned), let us introduce the following parameter

\begin{definition}
The following coefficient
\[
Cc=\frac{12+4\kappa-3s^{2}}{6+3\kappa-3s^{2}},
\]
will be called \textbf{continuity coefficient}. Above, as before, $s$ denotes
the skewness while $\kappa$ the excess kurtosis of the marginal measure $\mu$
of the $SMPR\left(  \left\{  \alpha_{n}\right\}  ,\mu\right)  $ that we are analyzing.
\end{definition}

Since, as mentioned earlier, if only the support of the marginal measure $\mu$
is unbounded the sequence $\left\{  \exp(-\alpha_{n}t)\right\}  $ must be a
moment sequence.

\begin{corollary}
If the support of the measure $\mu$ is unbounded and $3s^{2}\allowbreak
=\allowbreak2\kappa$ then the only $SMPR\left(  \left\{  \alpha_{n}\right\}
,\mu\right)  $ process having continuous path modifications is the one with
$\alpha_{n}\allowbreak=\allowbreak n\alpha_{1}$. From \cite{SzabStac} it
follows that such process is also a harness.
\end{corollary}

\begin{proof}
Notice that, following (\ref{a2m}). when $3s^{2}\allowbreak=\allowbreak
2\kappa$ we must have $\alpha_{2}\allowbreak=\allowbreak2\alpha_{1}$. Now, by
Lemma \ref{momenty} we have $\alpha_{n}\allowbreak=\allowbreak n\alpha_{1}.$
\end{proof}

Now let us show some examples and consider particular cases.

\section{Examples}

\begin{example}
[\textbf{Gaussian case}.]Although it was analyses at the beginning of the
paper, we are returning to this example, once more just to illustrate the
theory developed above on this very simple example of relatively simple
calculations. Let us consider $d\mu(x)\allowbreak=\allowbreak\exp
(-x^{2}/2)dx/\sqrt{2\pi}.$ Consequently, polynomials $h_{n}(x)$ are the
so-called probabilistic Hermite polynomials satisfying the following
three-term recurrence
\[
h_{n+1}(x)\allowbreak=\allowbreak xh_{n}(x)-nh_{n-1}(x),
\]
with $h_{-1}(x)\allowbreak=\allowbreak0$ and $h_{0}(x)\allowbreak
=\allowbreak1.$ It is a common knowledge, that for $\forall n\geq0:$%
\[
x^{j}\allowbreak=\allowbreak j!\sum_{m=0}^{\left\lfloor j/2\right\rfloor
}\frac{1}{2^{m}m!(j-2m)!}h_{j-2m}(x).
\]
Consequently, the numbers $c_{j,n}$ are given by%
\begin{equation}
c_{j,n}=\left\{
\begin{array}
[c]{ccc}%
0 & \text{if } & n>j\text{ or }j-n\text{ odd}\\
\frac{j!}{2^{(j-n)/2}((j-n)/2)!\sqrt{n!}} & \text{if} & j-n\text{ is even}%
\end{array}
\text{.}\right.  \label{GHjn}%
\end{equation}
There is $\sqrt{n!}$ in the denominator since polynomials $h_{n}$ are not
orthonormal. They become orthonormal after dividing $n-$th polynomial by
$\sqrt{n!}$. We have $\kappa\allowbreak=\allowbreak0$ and $s\allowbreak
=\allowbreak0$ so $\alpha_{2}\allowbreak=\allowbreak2\alpha_{1}.$ From the
Remark \ref{a2a} it follows that we must have $\alpha_{n}\allowbreak
=\allowbreak n\alpha_{1}$ and consequently we must be dealing with the
Ornstein -Uhlenbeck (OU) process. We can thus conclude that the
Ornstein-Uhlenbeck process is the only SMPR process with Gaussian marginals
that allow continuous path modification.
\end{example}

\begin{example}
[\textbf{Gamma distribution}.]Let us consider the Gamma distribution with rate
parameter zero and shape parameter $\beta>0,$ i.e., the distribution with the
following density:%
\[
f_{g}(x,\beta)\allowbreak=\allowbreak x^{\beta-1}\exp(-x)/\Gamma(\beta),
\]
for $x>0$ and $0$ otherwise. In order to simplify notation let us introduce
the so-called raising factorial, which is the following function:%
\[
(x)^{(n)}=x(x+1)\ldots(x+n-1),
\]
defined for all complex $x.$ Notice, that we have for all $x\neq0:$%
\[
(x)^{(n)}=\frac{\Gamma(x+n)}{\Gamma(x)}.
\]
It is well known that if $X\sim f_{g}(x,\beta),$ then $EX^{j}\allowbreak
=\allowbreak(\beta)^{(j)}.$ Thus we have $m_{2}\allowbreak=\allowbreak\beta,$
$m_{3}\allowbreak=\allowbreak2\beta,$ $m_{4}\allowbreak=\allowbreak
3\beta(2+\beta).$ Hence, the skewness parameter $s$ is equal to $2/\sqrt
{\beta}$ and the kurtosis $\kappa\allowbreak=\allowbreak3(2+\beta
)/\beta-3\allowbreak=\allowbreak6/\beta.$ Consequently, the parameter
$\alpha_{2}/\alpha_{1}$ is equal to%
\[
\frac{12+4\kappa-3s^{2}}{6+3\kappa-3s^{2}}=\frac{12+24/\beta-12/\beta
}{6+18/\beta-12/\beta}=2.
\]

Since the support of Gamma distribution is unbounded, we deduce that $\forall
n\geq1:\alpha_{n}\allowbreak=\allowbreak n\alpha_{1}$ and following
\cite{SzabStac}(Thm. 2) such process is also a harness.

Moreover, we know also that the polynomials $h_{n}(x|\beta)$ are proportional
to the so-called generalized Laguerre polynomials $L_{n}^{(\beta)},$ defined
by the following recurrence relation
\[
\allowbreak L_{n+1}^{(\beta)}=\frac{2n+\beta-x}{n+1}L_{n}^{(\beta)}%
-\frac{n+\beta-1}{n+1}L_{n-1}^{(\beta)},
\]
with $L_{-1}^{(\beta)}\allowbreak=\allowbreak0\allowbreak$ and $L_{0}%
^{(\beta)}\allowbreak=\allowbreak1.$ We also know that
\[
\frac{1}{\Gamma(\beta)}\int_{0}^{\infty}L_{n}^{(\beta}(x)L_{m}^{(\beta
)}(x)x^{\beta-1}\exp(-x)dx=\frac{\Gamma(n+\beta)}{n!}\delta_{nm}.
\]
Thus
\[
h_{n}(x|\beta)=\sqrt{\frac{n!}{\Gamma(n+\beta)}}L_{n}^{(\beta)}(x).
\]
Moreover, we also know the exact form of the two-dimensional density of such
$SMPR\left(  \left\{  n\alpha_{1}\right\}  ,f_{g}\right)  $ process. It has
complicated form and is commonly known under the name Hardy-Hille formula.
Namely, we have
\begin{gather*}
f_{g}(x|\beta)f_{g}(x|\beta)\sum_{j\geq0}h_{j}(x|\beta)h_{j}(y|\beta
)\exp(-n\alpha_{1}t)\\
=\frac{1}{(1-\exp(-\alpha_{1}t))(xy\exp(-\alpha_{1}t))^{(\beta-1)/2}}%
\exp(-(x+y)\\
\times\frac{\exp(-\alpha_{1}t)}{1-\exp(-\alpha_{1}t)})I_{\beta-1}\left(
\frac{2\exp(-\alpha_{1}t/2)\sqrt{xy}}{1-\exp(-\alpha_{1}t)}\right)  ,
\end{gather*}
where $I_{\alpha}$ denotes modified Bessel function of the first kind.
Recently, it has been shown in \cite{Szab22} that, as in the Gaussian case, we
have
\[
E(X_{\tau}-X_{\tau+t})^{2k}=\frac{(2k)!}{k!}\left(  \beta\right)  ^{\left(
k\right)  }(1-\exp(-\alpha_{1}t))^{k}.
\]
Hence, the gamma process with transition distribution given above is as smooth
as the Ornstein-Uhlenbeck process.
\end{example}

\begin{example}
[\textbf{Laplace distribution.}]\textbf{ }Laplace distribution is defined by
its density given by the formula $f_{L}(x)\allowbreak=\allowbreak
\exp(-\left\vert x\right\vert )/2,$ for $x\in\mathbb{R}$. It is also known
that it is a symmetric distribution, hence all its odd moments are equal to
$\ 0$ while all even moments, say of degree $2n$ are equal to $(2n)!.$
Consequently, we have $m_{2}\allowbreak=\allowbreak2,$ $m_{3}\allowbreak
=\allowbreak0$ and $m_{4}\allowbreak=\allowbreak24$, hence $\kappa
\allowbreak=\allowbreak3,$ and $s\allowbreak=\allowbreak0.$ Thus, we have
\[
\alpha_{2}/\alpha_{1}=\frac{12+4\times3}{6+3\times3}=8/5<2.
\]
Let us denote by $\left\{  h_{n}\right\}  _{n\geq0}$ the family of polynomials
orthonormal with respect to the measure with the density $f_{L}$. Thus, for
every moment sequence $\{1,\exp(-\alpha_{1}t),\allowbreak\exp(-\alpha
_{2}t),\ldots\},$ such that the bivariate function:%
\begin{equation}
1+\sum_{j\geq1}\exp(-\alpha_{j}t)h_{j}(x)h_{j}(y), \label{copula}%
\end{equation}
with $\alpha_{2}\allowbreak=\allowbreak8\alpha_{1}/5$ is nonnegative for all
$t,$ the $SMPR(\left\{  \alpha_{n}\right\}  ,f_{L})$ allows continuous path
modification with $E(X_{t+\tau}-X_{\tau})^{4}\allowbreak=\allowbreak
O(t^{2}).$

Interestingly, if we calculate (with the help of Mathematica) the coefficients
$\left\{  c_{i,n}\right\}  _{i=0,\ldots,6,n=0,\ldots,i},$ which can be easily
done using moments of the Laplace distribution, and then solving system of
equations (\ref{rown}) we get
\[
\alpha_{2}/\alpha_{1}=(35-\sqrt{105})/28,~\alpha_{3}/\alpha_{1}=(15-\sqrt
{105})/12.
\]
Consequently, every $SMPR(\left\{  \alpha_{n}\right\}  ,f_{L})$ with these
parameters $\alpha_{2},\alpha_{3}$ ( for which naturally function
(\ref{copula}) is nonnegative), would allow continuous path modification and
$E(X_{t+\tau}-X_{\tau})^{6}\allowbreak=\allowbreak O(t^{3}).$

We can continue this procedure. There are however some numerical problems
since there is not known a nice general form of polynomials that are
orthonormal with respect to the Laplace distribution.
\end{example}

\begin{example}
[\textbf{Beta distribution.}]So first let us start with the general case of
the Beta distribution. Under this name function, two distributions related to
one another. Namely, primarily the distribution with the density%
\[
f_{\beta}(x|\alpha,\beta)=x^{\gamma-1}(1-x)^{\beta-1}/B(\gamma,\beta),
\]
for $x\in(0,1)$ and $\gamma,\beta>0,$ where $B(\gamma,\beta)\allowbreak
=\allowbreak\Gamma(\gamma)\Gamma(\beta)/\Gamma(\gamma+\beta)$ is the so-called
beta function. Often under the name Beta distribution also works the
distribution with the following density%
\begin{equation}
f_{B}(x|\gamma,\beta)=2^{1-\gamma-\beta}(1+x)^{\gamma-1}(1-x)^{\beta
-1}/B(\gamma,\beta), \label{betaC}%
\end{equation}
for $x\in(-1,1)$ and $\gamma,\beta>0$. As one can easily notice if random
variable $X\sim f_{\beta}$ then $Y\allowbreak=\allowbreak2X-1$ has density
$f_{B}.$ It is also known that the so-called Jacobi polynomials are orthogonal
with respect to the measure with the density $f_{B}.$ In particular, we know
that
\[
f_{B}(x|1/2,1/2)\allowbreak=\allowbreak\frac{1}{\pi\sqrt{1-x^{2}}},
\]
that is, we are dealing with the so-called arcsine distribution while
$f_{B}(x|3/2,3/2)\allowbreak=\allowbreak\frac{2}{\pi}\sqrt{1-x^{2}}$ is known
under the name Wigner or semicircle distribution. It is known (see e.g.
\cite{Kotz95}) that the skewness of the beta distribution is equal to:%
\[
s=\frac{2(\beta-\gamma)\sqrt{(\gamma+\beta+1}}{(\gamma+\beta+2)\sqrt
{\gamma\beta}},
\]
while the excess kurtosis is given by the formula;%
\[
\kappa=\frac{6\left(  (\gamma-\beta\right)  ^{2}(\gamma+\beta+1)-(\gamma
+\beta+2)\gamma\beta)}{\gamma\beta(\gamma+\beta+2)(\gamma+\beta+3)}.
\]

Now
\begin{equation}
\frac{12+4\kappa-3s^{2}}{6+3\kappa-3s^{2}}=\frac{2(\gamma+\beta+1)}%
{\gamma+\beta}=2+\frac{2}{\gamma+\beta}. \label{war}%
\end{equation}

\end{example}

\begin{example}
[\textbf{Arcsine distribution.}]\textbf{ }Arcsine distribution is an example
of the beta distribution with parameters $\gamma\allowbreak=\allowbreak1/2,$
$\beta\allowbreak=\allowbreak1/2.$ Let us recall that in this case, the
sequence of polynomials orthogonal with respect to the measure with the
density $f_{B}(1/2,1/2)$ are the so-called Chebyshev polynomials of the first
kind $\left\{  T_{n}(x)\right\}  _{n\geq-1}$. Following \cite{Mas2003} let us
recall the most important properties of these polynomials. They are defined by
the following three-term recurrence
\[
2xT_{n}(x)=T_{n+1}(x)+T_{n-1}(x),
\]
with $T_{0}(x)\allowbreak=\allowbreak1$ and $T_{1}(x)\allowbreak=\allowbreak
x.$ What is however important for our purposes, is the following property of
polynomials $\left\{  T_{n}\right\}  _{n\geq0}.$ Namely, we have:%
\[
T_{n}(\cos\varphi)=\cos(n\varphi).
\]
for $n\geq0$ and $\varphi\in\mathbb{R}$. Moreover, it is known that
\[
\frac{1}{\pi}\int_{-1}^{1}T_{n}(x)T_{m}(x)\frac{1}{\sqrt{1-x^{2}}}dx=\left\{
\begin{array}
[c]{ccc}%
1 & \text{if} & n=m=0\\
\delta_{mn}/2 & \text{if } & n+n>0
\end{array}
\right.  .
\]
Thus, polynomials are now given by $h_{n}(x)\allowbreak=\allowbreak\sqrt
{2}T_{n}(x)$ for $n\geq1.$

As one can see in order to have a continuous path modification of the
$SMPR(\left\{  \alpha_{n}\right\}  ,\allowbreak f_{B}(1/2,1/2))$ one has to
have $\alpha_{2}\allowbreak=\allowbreak4\alpha_{1}.$ We will define the
sequence of numbers $\left\{  \alpha_{n}\right\}  _{n\geq1}$ that satisfies
the above-mentioned conditions and above all will lead to a nonnegative
transition density. To do it, let us recall the definition of the Jacobi Theta
function. In fact, there exist different definitions of this function. The
differences are minor and concern details. For the clarity of notation let us
select the one provided by Wolfram MathWorld with a slight modification.
Namely, we define for $\left\vert q\right\vert \,<1$ and complex $\alpha:$
\[
\theta(q;\alpha)\allowbreak=\allowbreak\sum_{j=-\infty}^{\infty}q^{j^{2}}%
\exp(2ij\alpha).
\]
One can easily notice, that
\begin{equation}
\theta(q;\alpha)=1+2\sum_{j=1}^{\infty}q^{j^{2}}\cos(2j\alpha). \label{theta}%
\end{equation}
Let us also recall the following so-called triple product identity (see, e.g.,
\cite{Abr1966}) that reads:%
\begin{equation}
\theta(q;\alpha)=\prod_{m=1}^{\infty}(1-q^{2m})(1+2\cos(2\alpha)q^{2m-1}%
+q^{4m-2}). \label{tpi}%
\end{equation}
Thus we deduce that for $\left\vert q\right\vert <1$ and $\alpha\in\mathbb{R}%
$, w have $\theta(q;\alpha)>0$. Consequently, we have in terms of the new
variables $\varphi$ and $\phi:$ $2T_{n}(\cos\varphi)T_{n}(\cos\phi
)\allowbreak=\allowbreak2\cos(n\varphi)\cos(n\phi)\allowbreak=\allowbreak
\cos(n(\varphi-\phi))\allowbreak+\allowbreak\cos(n(\varphi+\phi))$ and
\begin{gather*}
1+2\sum_{n\geq1}\rho^{n^{2}}T_{n}(\cos\varphi)T_{n}(\cos\phi)\\
=1+\sum_{n\geq1}\rho^{n^{2}}\cos(n(\varphi-\phi))+\sum_{n\geq1}\rho^{n^{2}%
}\cos(n(\varphi+\phi))\\
=(\theta(\rho;((\varphi-\phi))+\theta(\rho;((\varphi+\phi)))/2.
\end{gather*}
Thus, returning to variables $x$ and $y,$ we get:
\begin{gather}
1+2\sum_{n\geq1}\rho^{n^{2}}T_{n}(x)T_{n}(y)=\label{2wymT}\\
\frac{1}{2}\left(  \theta(\rho;(\arccos(x)-\arccos(y))+\theta(\rho
;(\arccos(x)+\arccos(y))\right)  .\nonumber
\end{gather}
Setting $\rho\allowbreak=\allowbreak\exp(-\alpha_{1}t)$ in \ref{2wymT} and
multiplying by, say $\frac{1}{\pi\sqrt{1-y^{2}}},$ we get a transitional
density of $X_{t+\tau}$ given $X_{\tau}\allowbreak=\allowbreak y.$ Moreover,
we have for some $\alpha_{1}>0$, all $n\geq0$ and almost all (mod $\mu)$
$y\in\lbrack-1,1]$ the following relationships%
\[
E(T_{n}(X_{t+\tau})|X_{\tau}=y)=\exp(-\alpha_{1}tn^{2})T_{n}(y).
\]

\begin{remark}
Let us note that we have $E(X_{t+\tau}-X_{\tau})^{4}\allowbreak=\allowbreak
O(t^{2})$ as it follows directly from (\ref{war}). However, we know the
parameters $c_{i,n}$ defined above since we have:%
\[
x^{j}\allowbreak=\allowbreak2^{1-j}\sum_{\substack{n=0\\j\equiv
n,~\operatorname{mod}2}}^{\prime}\binom{j}{(j-n)/2}h_{n}(x)\overset{df}{=}%
\sum_{n=0}^{j}c_{i,n}h_{n}(x),
\]
where the $^{\prime}$ above the sum means that when $n\allowbreak
=\allowbreak0$ then the appropriate coefficient is divided by $2.$ Thus
consequently, we can, using Mathematica, check that for $k\allowbreak
=\allowbreak3,4,5$ we have also $E(X_{t+\tau}-X_{\tau})^{2k}\allowbreak
=\allowbreak O(t^{k}).$ Thus, one can express the following conjecture:

\begin{conjecture}
Let us consider $\mathbf{X\allowbreak=\allowbreak}SMPR\left(  \left\{
n^{2}\alpha_{1}\right\}  ,f_{B}(x|1/2,1/2\right)  ).$ Then for all $k\geq1$ we
have $E(X_{t+\tau}-X_{\tau})^{2k}\allowbreak=\allowbreak O(t^{k}).$
\end{conjecture}

Thus, this process is as 'smooth' as the Ornstein-Uhlenbeck process. It is
not, however, a harness.
\end{remark}

\begin{remark}
Notice also, that the sequence $\left\{  \exp(-n^{2}\alpha_{1})\right\}
_{n\geq0}$ is not a moments' sequence. This is so, for example, the matrix
$\left[
\begin{tabular}
[c]{ll}%
$1$ & $\exp(-a_{1}t)$\\
$\exp(-a_{1}t)$ & $\exp(-4a_{1}t)$%
\end{tabular}
\ \right]  $ is not nonnegative definite. Recall, that this feature of
$SMPR\left(  \left\{  \alpha_{n}\right\}  ,\mu\right)  $, as pointed out in
\cite{Szabl21}, has been allowed because the support of the measure $\mu$ is
bounded in this case.
\end{remark}
\end{example}

\begin{example}
[\textbf{Semicircle distribution.}]\textbf{ }As mentioned above, the
semicircle or Wigner distribution is another example of a beta distribution,
this time we with $\alpha\allowbreak=\allowbreak3/2$ and $\beta\allowbreak
=\allowbreak3/2$ which leads to the following distribution with the density
\[
f_{B}(x|3/2,3/2)\allowbreak=f_{W}(x)=\allowbreak\frac{2}{\pi}\sqrt{1-x^{2}}.
\]
Further, we know that the family of polynomials that are orthogonal with
respect to the Wigner measure is the family of Chebyshev polynomials of the
second kind , that is the family $\left\{  U_{n}(x)\right\}  _{n\geq0},$
satisfying the following three-term recurrence
\[
2xU_{n}(x)=U_{n-1}(x)+U_{n+1}(x),
\]
with $U_{-1}(x)\allowbreak=\allowbreak0,$ $U_{0}(x)\allowbreak=\allowbreak1.$
Besides, we also have:%
\[
\int_{-1}^{1}U_{n}(x)U_{m}(x)f_{W}(x)dx=\delta_{nm}.
\]
Hence, the family $\left\{  U_{n}\right\}  $ constitutes an orthonormal family.

Now notice that setting $\gamma\allowbreak=\allowbreak\beta\allowbreak
=\allowbreak3/2$ in (\ref{war}) we get $Cc\allowbreak=\allowbreak8/3.$ Further
finding that
\begin{align*}
x  &  =U_{1}(x)/2,~x^{2}=\frac{1}{2^{2}}(U_{2}(x)+1),~x^{3}=\frac{1}{2^{3}%
}(U_{3}(x)+2U_{1}(x)),\\
x^{4}  &  =\frac{1}{2^{4}}(U_{4}(x)+3U_{2}(x)+2),~x^{5}=\frac{1}{2^{5}}%
(U_{5}(x)+4U_{3}(x)+5U_{1}(x)),\\
x^{6}  &  =\frac{1}{2^{6}}(U_{6}(x)+5U_{4}(x)+9U_{2}(x)+5),\\
x^{7}  &  =\frac{1}{2^{7}}(U_{7}(x)+6U_{5}(x)+14U_{3}(x)+14U_{1}(x)),\\
x^{8}  &  =\frac{1}{2^{8}}(U_{8}(x)+7U_{6}(x)+20U_{4}(x)/2^{6}+28U_{2}(x)+14),
\end{align*}
we can, using formula (\ref{rown}), find that $\alpha_{3}/\alpha
_{1}\allowbreak=\allowbreak5,$ $\alpha_{4}/\alpha_{1}\allowbreak
=\allowbreak8.$ So generalizing we will consider the following function
\[
f(x,y|\alpha)=\frac{4}{\pi^{2}}\sqrt{(1-x^{2})(1-y^{2})}\sum_{n\geq0}%
\exp(-\alpha n(n+2)/3)U_{n}(x)U_{n}(y),
\]
as the candidate for the density of the bivariate measure $\rho$ according to
the formula (\ref{Lanc}). First of all, notice that $\int_{-1}^{1}\int%
_{-1}^{1}f(x,y|\alpha)dxdy\allowbreak=\allowbreak1$ for all $\alpha>0.$ Thus
is remains to show that $f(x,y|\alpha)\geq0$ for all $x,y\in\lbrack-1,1]$. in
order to simplify notation let us consider an auxiliary function:%
\[
g(x,y|\rho)=\sum_{n\geq0}\rho^{n(n+2)}U_{n}(x)U_{n}(y).
\]
We immediately notice that, since we have $n(n+2)\allowbreak=\allowbreak
(n+1)^{2}-1,$ $U_{n}(\cos x)\allowbreak=\allowbreak\frac{\sin(n+1)x}{\sin x}$
and $2\sin(\gamma)\sin(\phi)\allowbreak=\allowbreak\cos(\gamma-\phi
)\allowbreak-\allowbreak\cos(\gamma+\phi),$ that
\begin{align*}
g(\cos\gamma,\cos\phi|\rho)  &  =\frac{1}{4\rho\sin(\gamma)\sin(\phi)}\left(
2\sum_{n=1}\rho^{n^{2}}\cos(n(\gamma-\phi))-2\sum_{n=1}\rho^{n^{2}}%
\cos(n(\gamma+\phi))\right) \\
&  =\frac{1}{4\rho\sin(\gamma)\sin(\phi)}\left(  \left(  \theta(\rho
;(\gamma-\phi)/2\right)  -\theta(\rho;(\gamma+\phi)/2)\right)  .
\end{align*}
Now notice that the following (\ref{tpi}) function $\theta(\rho;\gamma/2)$
behaves like a cosine function for positive $\rho$ and $\gamma\in(0,\pi)$ in
the sense that it increases and decreases on exactly the same intervals. In
particular, it means that both functions have derivatives of the same signs.
This leads to the conclusion that
\[
\frac{\theta(\rho;\gamma/2)-\theta(\rho;\phi/2)}{\cos\gamma-\cos\phi}\geq0,
\]
for $\rho\in(0,1)$ and $\gamma,\phi\in(0,\pi).$ But we have%
\[
g(\cos\gamma,\cos\phi|\rho)=\frac{\left(  \theta(\rho;(\gamma-\phi)/2\right)
-\theta(\rho;(\gamma+\phi)/2)}{2\rho(\cos(\gamma-\phi)-\cos(\gamma+\phi))}.
\]
Returning to variables $x$ and $y$ we get:%
\begin{gather*}
g(x,y|\rho)=\left(  \left(  \left(  \theta(\rho;(\arccos(x)-\arccos
(y))/2\right)  -\theta(\rho;(\arccos(x)+\arccos(y))/2\right)  \right) \\
/\left(  4\rho\sqrt{(1-x^{2})(1-y^{2})}\right)  ,
\end{gather*}

consequently, transition probabilities density is equal to:%
\[
\eta(x|y,t)\allowbreak=\allowbreak\frac{\left(  \left(  \theta(e^{-\alpha
t};(\arccos(x)-\arccos(y))/2\right)  -\theta(e^{-\alpha t};(\arccos
(x)+\arccos(y))/2\right)  }{4\exp(-\alpha t)\sqrt{1-y^{2}}},
\]
and we have the following relationships
\[
E(U_{n}(X_{t+\tau})|X_{t}\allowbreak=\allowbreak y)=\exp(-\alpha
tn(n+2))U_{n}(y),
\]
for some $\alpha>0,$ $n\geq0$ and almost all (mod $\mu)$ $y\in
\operatorname*{supp}\mu$.
\end{example}

\begin{remark}
Let us note that we have $E(X_{t+\tau}-X_{\tau})^{4}\allowbreak=\allowbreak
O(t^{2})$ as it follows directly from (\ref{war}). However, we can find
$c_{i,n}$ defined above for several $i$ and $n$ numerically performing
Cholesky decomposition of the moment matrix. The $2n$-th moment of the
semicircle distribution is know and is equal to, depending on the radius $r,$
is equal $r^{2n}C_{n}/4^{n}.$ Here $C_{n}$ denotes $n-$th Catalan number and
by the Wigner distribution with radius $r$ we mean the one with the following
density defined for $\left\vert x\right\vert \leq r:$%
\[
f_{W}(x|r)=\frac{2}{\pi r^{2}}\sqrt{r^{2}-x^{2}}.
\]
Thus consequently, we can, using Mathematica, check that for $k\allowbreak
=\allowbreak3,4,5$ that also $E(X_{t+\tau}-X_{\tau})^{2k}\allowbreak
=\allowbreak O(t^{k}).$ Thus one can utter the following conjecture:
\end{remark}

\begin{conjecture}
Let us consider $\mathbf{X\allowbreak=\allowbreak}SMPR\left(  \left\{
n(n+2)\alpha\right\}  ,f_{B}(x|3/2,3/2\right)  ).$ Then for all $k\geq1$ we
have $E(X_{t+\tau}-X_{\tau})^{2k}\allowbreak=\allowbreak O(t^{k}).$
\end{conjecture}

\begin{example}
[$q-$\textbf{Normal distribution.}]\textbf{ }$q-$Normal is in fact a family of
distributions indexed by a parameter $q\in\lbrack-1,1]$. Its properties have
been described in many papers in particular the following review paper
\cite{Szab2020} where this distribution and the others related to it are
presented and various their applications in probability theory and
combinatorics are presented. It is defined as follows.\newline For $q=-1,$ it
is a discrete $2-$ point distribution, which assigns values $1/2$ to $-1$ and
$1$.\newline For $q\in\left(  -1,1\right)  ,$ it has density given by
\[
f_{N}(x|q)=\frac{\sqrt{1-q}}{2\pi\sqrt{4-(1-q)x^{2}}}\prod_{k=0}^{\infty
}\left(  (1+q^{k})^{2}-(1-q)x^{2}q^{k}\right)  \prod_{k=0}^{\infty}%
(1-q^{k+1}),
\]
for $\left\vert x\right\vert \leq\frac{2}{\sqrt{1-q}}$. In particular
$f_{H}\left(  x|0\right)  \allowbreak=\allowbreak\frac{1}{2\pi}\sqrt{4-x^{2}%
},$ for $|x|\leq2.$ Hence, it is a Wigner distribution with a radius $2$. For
$q=1,$ the $q$\emph{-Gaussian} distribution is the Normal distribution with
parameters $0$ and $1.$

It is also known that the family of polynomials orthogonal with respect to
$q-$Normal distribution are the so-called $q-$Hermite polynomials satisfying
the following three-term recurrence
\[
H_{n+1}(x|q)=xH_{n}(x|q)-[n]_{q}H_{n-1}(x|q),
\]
with $H_{-1}(x|q)\allowbreak=\allowbreak0$ and $H_{0}(x|q)\allowbreak
=\allowbreak1,$ where $[n]_{q}\allowbreak=\allowbreak1+\ldots+q^{n-1}$. It is
also known that
\[
\int_{-2/\sqrt{1-q}}^{2/\sqrt{1-q}}H_{n}(x|q)H_{m}(x|q)f_{N}(x|q)dx=\delta
_{nm}[n]_{q}!,
\]
where $[n]_{q}!\allowbreak=\allowbreak\prod_{i=0}^{n}[i]_{q}$. Recall that
$-1<q\leq1$ and that by the $q-$Normal distribution, we mean the one defined
by the following density. Following \cite{SzabMom} (Proposition 2) we know
that $m_{4}(q)\allowbreak=\allowbreak2+q$, $m_{3}(q)\allowbreak=\allowbreak0$
and $m_{2}(q)\allowbreak=\allowbreak1.$ Hence, $s\allowbreak=\allowbreak0$ and
$\kappa\allowbreak=\allowbreak q-1.$ Consequently, we have
\[
Cc=\alpha_{2}/\alpha_{1}=\frac{12+4(q-1)}{6+3(q-1)}=\frac{8+4q}{3+3q}.
\]
This parameter is greater than $2$ for all $q<1.$ Recall that $q\allowbreak
=\allowbreak1$ refers to the Gaussian case that was already analyzed. Then, as
it can be seen we have $Cc\allowbreak=\allowbreak2.$ Since for all cases other
than $1$ the possible values of parameter $q,$ the support of the measure
$f_{N}$ is bounded, equal to $S_{q}\allowbreak=\allowbreak\lbrack-2/\sqrt
{1-q},2/\sqrt{1-q}]$ the sequence $\left\{  \exp(-\alpha_{n}t)\right\}  $ may
not be a moment sequence. Hence, one can expect that for all choices of the
sequence $\left\{  \alpha_{n}\right\}  $ such that $\alpha_{2}\allowbreak
=\allowbreak\frac{8+4q}{3+3q}\alpha_{1}$ we can expect path continuity of the
given SMPR process. The only thing is to select the sequence of positive
numbers $\left\{  \alpha_{n}\right\}  _{n\geq1}$ with $\alpha_{2}$ given above
in such a way that the function given by (\ref{gest}) is nonnegative for all
$\left\vert x\right\vert ,\left\vert y\right\vert \leq\frac{2}{\sqrt{1-q}}$.

Note that the case $q\allowbreak=\allowbreak0$ refers to the Wigner case also
analyzed above (although with a different radius). What is more, as shown in
\cite{Bo} to each classical $q$-Normal (Gaussian) process one can define a
related process defined within the non-commutative probability setting. Among
those related processes defined within the non-commutative probability
setting, the case $q\allowbreak=\allowbreak0$ refers to the so-called free
probability. Thus, one can expect that our analysis concerning Wigner marginal
distribution presented above can help to construct free probability process
having a continuous path modification.
\end{example}


\begin{thebibliography}{99}                                                                                               %


\bibitem {Alexits61}G. Alexits, \emph{Convergence problems of orthogonal
series.} Translated from the German by I. F\"{o}lder. International Series of
Monographs in Pure and Applied Mathematics, Vol. \textbf{20} Pergamon Press,
New York-Oxford-Paris 1961 \{%
$\backslash$%
rm ix\}+350 pp. MR0218827

\bibitem {Bo}M. Bo\.{z}ejko, B. K\"{u}mmerer, and R. Speicher, Roland.
\emph{\$q\$-Gaussian processes: non-commutative and classical aspects.} Comm.
Math. Phys. 185 (1997), no. 1, 129--154. MR1463036 (98h:81053)

\bibitem {brwe05}W. Bryc, J. Weso\l owski, \emph{Conditional Moments of }%
$q$\emph{-Meixner Processes}, Probab. Theory Rel. Fields ,\textbf{131}(2005), 415-441

\bibitem {BryWe}W. Bryc, J. Weso\l owski, Bi - Poissson process, Infinite
Dimensional Analysis, \emph{Quantum Probability and Related Topics}
\textbf{10}(2007), (2) , 277-291

\bibitem {BryMaWe}W. Bryc, W. Matysiak, J. Weso\l owski, \emph{The Bi -
Poisson process: a quadratic harness.} Annals of Probability \textbf{36} (2)
(2008), s. 623-646

\bibitem {BryMaWe07}W. Bryc, W. Matysiak, J. Weso\l owski, \emph{Quadratic
Harnesses, }$q-$\emph{commutations, and orthogonal martingale polynomials.}
Trans. Amer. Math. Soc. \textbf{359} (2007), no. 11, 5449--5483

\bibitem {BryWe10}W. Bryc, J. Weso\l owski, \emph{Askey-Wilson polynomials,
quadratic harnesses and martingales.} Ann. Probab. \textbf{38} (2010), no. 3,
1221--1262. MR2674998

\bibitem {BryMaWe11}W. Bryc, W. Matysiak, J. Weso\l owski, \emph{Free
Quadratic Harness}. Stochastic Processes and their Applications\emph{
}\textbf{121} (2011) 657--671

\bibitem {BryWe12}W. Bryc, J. Weso\l owski, \emph{Stitching pairs of Levy
processes into harnesses}, Stochastic Processes and their Applications,
\textbf{122}(2012), 2854--2869

\bibitem {Abr1966}Handbook of mathematical functions, with formulas, graphs,
and mathematical tables. Edited by Milton Abramowitz and Irene A. Stegun
\emph{Dover Publications, Inc.,} New York 1966 +1046 pp. MR0208797

\bibitem {Cuch12}Ch. Cuchiero, M. Keller-Ressel, J. Teichmann,
\emph{Polynomial processes and their applications to mathematical finance.}
Finance Stoch. 16 (2012), no. 4, 711--740. MR2972239

\bibitem {Ham67}J.M. Hammersley, \emph{Harnesses}. 1967 Proc. Fifth Berkeley
Sympos. Mathematical Statistics and Probability (Berkeley, Calif., 1965/66),
Vol. III: Physical Sciences pp. 89--117 Univ. California Press, Berkeley,
Calif. MR0224144 (36 \#7190)

\bibitem {Kotz95}N.L. Johnson, S. Kotz, and N. Balakrishnan, (1995).
\emph{Continuous Univariate Distributions} Vol. 2 (2nd ed.). Wiley, "Chapter
25:Beta Distributions"

\bibitem {Mas2003}J.C. Mason, D. C. Handscomb, \emph{Chebyshev Polynomials.}
Chapman \& Hall-CRC, Boca Raton, FL, 2003. xiv+341 pp. ISBN: 0-8493-0355-9 MR1937591

\bibitem {Lancaster58}H. O. Lancaster, \emph{The structure of bivariate
distributions,} Ann. Math. Statistics, vol. 29, no. 3, pp. 719-736, September 1958.

\bibitem {Lancaster63(2)}H. O. Lancaster, \emph{Correlation and complete
dependence of random variables,} Ann. Math. Statistics, vol. 34, no. 4, pp.
1315-1321, December 1963.

\bibitem {Lancaster63(1)}H. O. Lancaster, \emph{Correlations and canonical
forms of bivariate distributions,} Ann. Math. Statistics, vol. 34, no. 2, pp.
532-538, June 1963.

\bibitem {Lancaster75}Lancaster, H. O. \emph{Joint probability distributions
in the Meixner classes.} J. Roy. Statist. Soc. Ser. B 37 (1975), no. 3,
434--443. MR0394971 (52 \#15770)

\bibitem {Ray56}D. Ray, \emph{Stationary Markov processes with continuous
paths.} Trans. Amer. Math. Soc. 82 (1956), 452--493. MR0102857

\bibitem {Strook79}D. Stroock, and S.R.S. Varadhan, \emph{Multidimensional
diffusion processes.} Grundlehren der Mathematischen Wissenschaften
[Fundamental Principles of Mathematical Sciences], 233. Springer-Verlag,
Berlin-New York, 1979. \{%
$\backslash$%
rm xii\}+338 pp. ISBN: 3-540-90353-4 MR0532498

\bibitem {Sim98}B. Simon, \emph{The classical moment problem as a self-adjoint
finite difference operator.} Adv. Math. 137 (1998), no. 1, 82--203. MR1627806 (2001e:47020),

\bibitem {Sriv72}H.M. Srivastava, J.P. Singhal, S\emph{ome extensions of the
Mehler formula.} Proc. Amer. Math. Soc. 31 (1972), 135--141. MR0285738

\bibitem {Szab-OU-W}P. J. Szab\l owski, $q-$\emph{Wiener and }$(\alpha
,q)-$\emph{ Ornstein--Uhlenbeck processes. A generalization of known
processes,} Theory of Probability and Its Applications, 56\textbf{ }(4), 2011,
742--772, http://arxiv.org/abs/math/0507303

\bibitem {SzabChol}P.J. Szab\l owski, \emph{A few remarks on orthogonal
polynomials}, Appl. Math. Comput. \textbf{252} (2015), 215--228. http://arxiv.org/abs/1303.0627

\bibitem {SzablPoly}P.J. Szab\l owski, \emph{On Markov processes with
polynomial conditional moments.} Trans. Amer. Math. Soc. 367 (2015), no. 12,
8487--8519. MR3403063 http://arxiv.org/abs/1210.6055

\bibitem {SzabMom}P.J. Szab\l owski, Moments of \$q\$-normal and conditional
\$q\$-normal distributions. \emph{Statist. Probab. Lett.} \textbf{106} (2015),
65--72. MR3389972

\bibitem {SzabStac}P.J. Szab\l owski, \emph{On stationary Markov processes
with polynomial conditional moments.} Stoch. Anal. Appl. 35 (2017), no. 5,
852--872. MR3686472

\bibitem {SzabMarkov18}P.J. Szab\l owski, \emph{Markov processes, polynomial
martingales and orthogonal polynomials.} Stochastics 90 (2018), no. 1, 61--77. MR3750639

\bibitem {Szab2020}P.J. Szab\l owski, \emph{On the families of polynomials
forming a part of the Askey-Wilson scheme and their probabilistic
applications.} Infin. Dimens. Anal. Quantum Probab. Relat. Top. 25 (2022), no.
1, Paper No. 2230001, 57 pp. MR4408180

\bibitem {Szabl21}P.J. Szab\l owski, \emph{On positivity of orthogonal series
and its applications in probability}, Positivity\emph{ }26\textbf{, }article
19(2022\textbf{)}, https://arxiv.org/abs/2011.02710.

\bibitem {Szab22}P.J. Szab\l owski, \emph{Some polynomial identities involving
binomial coefficients, double and rising factorials and their probabilistic
interpretations and proofs,} http://arxiv.org/abs/2206.02201, submitted

\bibitem {SzabDif22}P.J. Szab\l owski, \emph{Moment sequences and difference
equations,} ArXiv:https://arxiv.org/abs/2204.04706, submitted
\end{thebibliography}
\end{document}